\documentclass[12pt]{amsart}
\usepackage[normalem]{ulem}
\usepackage{url}

\usepackage{verbatim}
\usepackage{amssymb}
\usepackage{upref}
\usepackage[all]{xy}
\usepackage[usenames,dvipsnames]{color}
\usepackage{url}
\usepackage[normalem]{ulem}

\allowdisplaybreaks

\newtheorem{thm}{Theorem}[section]
\newtheorem*{thm*}{Theorem}
\newtheorem{lem}[thm]{Lemma}
\newtheorem{cor}[thm]{Corollary}
\newtheorem{prop}[thm]{Proposition}
\newtheorem{obs}[thm]{Observation}

\theoremstyle{definition}
\newtheorem{defn}[thm]{Definition}
\newtheorem{q}[thm]{Question}

\newtheorem{ex}[thm]{Example}
\newtheorem{notn}[thm]{Notation}

\newtheorem*{notn*}{Notation}

\newtheorem*{hyp*}{Hypothesis}

\newtheorem{rem}[thm]{Remark}
\newtheorem*{rem*}{Remark}
\newtheorem{rems}[thm]{Remarks}

\numberwithin{equation}{section}

\newcommand{\secref}[1]{Section~\textup{\ref{#1}}}

\newcommand{\thmref}[1]{Theorem~\textup{\ref{#1}}}
\newcommand{\corref}[1]{Corollary~\textup{\ref{#1}}}
\newcommand{\lemref}[1]{Lemma~\textup{\ref{#1}}}
\newcommand{\propref}[1]{Proposition~\textup{\ref{#1}}}

\newcommand{\defnref}[1]{Definition~\textup{\ref{#1}}}
\newcommand{\remref}[1]{Remark~\textup{\ref{#1}}}
\newcommand{\qref}[1]{Question~\textup{\ref{#1}}}

\newcommand{\midtext}[1]{\quad\text{#1}\quad}
\newcommand{\righttext}[1]{\quad\text{#1 }}
\renewcommand{\and}{\midtext{and}}

\newcommand{\C}{\mathbb C}

\newcommand{\F}{\mathbb F}

\newcommand{\KK}{\mathcal K}

\newcommand{\JJ}{\mathcal J}

\newcommand{\LL}{\mathcal L}

\newcommand{\EE}{\mathcal E}

\renewcommand{\AA}{\mathcal A}

\renewcommand{\epsilon}{\varepsilon}

\DeclareMathOperator{\ad}{Ad}

\DeclareMathOperator{\ind}{Ind}

\DeclareMathOperator*{\spn}{span}
\DeclareMathOperator*{\clspn}{\overline{\spn}}

\newcommand{\id}{\text{\textup{id}}}

\newcommand{\<}{\langle}
\renewcommand{\>}{\rangle}
\renewcommand{\iff}{\ensuremath{\Leftrightarrow}}

\newcommand{\inv}{^{-1}}
\newcommand{\normal}{\vartriangleleft}

\renewcommand{\bar}{\overline}
\newcommand{\what}{\widehat}
\newcommand{\wilde}{\widetilde}

\newcommand{\ann}{^\perp}
\newcommand{\pann}{{}\ann}

\newcommand{\cs}%
{\ensuremath{\mathbf{C^*}}}
\newcommand{\csnd}%
{\ensuremath{\cs\!\!_\mathbf{nd}}}
\newcommand{\coact}%
{\ensuremath{\mathbf{C^*coact}}}
\newcommand{\coactnd}%
{\ensuremath{\coact_\mathbf{nd}}}
\newcommand{\coactn}%
{\ensuremath{\coact^\mathbf{n}}}
\newcommand{\coactnnd}%
{\ensuremath{\coactn_\mathbf{nd}}}
\newcommand{\coactm}%
{\ensuremath{\coact^\mathbf{m}}}
\newcommand{\coactmnd}%
{\ensuremath{\coactm_\mathbf{nd}}}

\begin{document}
\title{Exotic coactions}

\author[Kaliszewski]{S. Kaliszewski}
\address{School of Mathematical and Statistical Sciences
\\Arizona State University
\\Tempe, Arizona 85287}
\email{kaliszewski@asu.edu}

\author[Landstad]{Magnus~B. Landstad}
\address{Department of Mathematical Sciences\\
Norwegian University of Science and Technology\\
NO-7491 Trondheim, Norway}
\email{magnusla@math.ntnu.no}

\author[Quigg]{John Quigg}
\address{School of Mathematical and Statistical Sciences
\\Arizona State University
\\Tempe, Arizona 85287}
\email{quigg@asu.edu}

\subjclass[2000]{Primary  46L05}

\keywords{group $C^*$-algebra, coaction, $C^*$-bialgebra, Fourier-Stieltjes algebra}

\date{\today}

\begin{abstract}
If a locally compact group $G$ acts on a $C^*$-algebra B, we have both full and reduced crossed products, and each has a coaction of $G$.
We investigate ``exotic'' coactions in between,
that are determined by certain ideals $E$ of the Fourier-Stieltjes algebra $B(G)$ --- an approach that is inspired by recent work of Brown and Guentner on new $C^*$-group algebra completions.
We actually carry out the bulk of our investigation in the general context of coactions on a $C^*$-algebra $A$.
Buss and Echterhoff have shown that not every coaction comes from one of these ideals, but nevertheless the ideals do generate a wide array of exotic coactions.
Coactions determined by these ideals $E$ satisfy a certain ``$E$-crossed product duality'', intermediate between full and reduced duality.
We give partial results concerning exotic coactions, with the ultimate goal being a classification of which coactions are determined by ideals of $B(G)$.
\end{abstract}
\maketitle

\section{Introduction}
\label{intro}

If $\alpha$ is an action of a nonamenable locally compact group $G$ on a $C^*$-algebra $B$, there are in general numerous crossed product $C^*$-algebras; the largest is the full crossed product $B\rtimes_\alpha G$, and the smallest is the reduced crossed product $B\rtimes_{\alpha,r} G$. But there are frequently many ``exotic'' crossed products in between,
i.e., quotients $(B\rtimes_\alpha G)/I$ where $I$ is an ideal contained in the kernel of the regular representation $\Lambda:B\rtimes_\alpha G\to B\rtimes_{\alpha,r} G$.
A na\"ive question is: how to classify these ``large quotients'' of the crossed product?
This is surely too large a class to seriously contemplate.
We are interested in the large quotients that carry a ``dual coaction'' $\delta$, as indicated in the commutative diagram
\[
\xymatrix@C-20pt{
B\rtimes_\alpha G \ar[rr]^-{\what\alpha} \ar[dd]_\Lambda \ar[dr]^q
&&B\rtimes_\alpha G\otimes C^*(G) \ar'[d]^{\Lambda\otimes\id}[dd] \ar[dr]^{q\otimes\id}
\\
&(B\rtimes_\alpha G)/I \ar[rr]^(.3)\delta \ar[dl]
&&(B\rtimes_\alpha G)/I\otimes C^*(G) \ar[dl]
\\
B\rtimes_{\alpha,r} G \ar[rr]_-{\what\alpha^n}
&&B\rtimes_{\alpha,r} G\otimes C^*(G).
}
\]
We ask, how to classify these \emph{exotic coactions}?

Motivated by a recent paper of Brown and Guentner \cite{BrownGuentner}, we introduce a tool that produces many (but not all --- see below) of these exotica. To clarify matters, consider the special case $B=\C$, so we have a diagram
\[
\xymatrix@C-20pt{
C^*(G) \ar[rr]^-{\delta_G} \ar[dd]_\lambda \ar[dr]^q
&&C^*(G)\otimes C^*(G) \ar'[d]^{\lambda\otimes\id}[dd] \ar[dr]^{q\otimes\id}
\\
&C^*(G)/I \ar[rr]^(.3)\delta \ar[dl]
&&C^*(G)/I\otimes C^*(G) \ar[dl]
\\
C^*_r(G) \ar[rr]_-{\delta_G^n}
&&C^*_r(G)\otimes C^*(G)
}
\]
Then $I\subset \ker\lambda$, and in \cite[Corollary~3.13]{graded} we proved that
a large quotient $C^*(G)/I$ carries a coaction if and only if the annihilator $E=I^\perp$ in the Fourier-Stieltjes algebra $B(G)=C^*(G)^*$ is an ideal,
which will necessarily be \emph{large} in the sense that it
contains the reduced Fourier-Stieltjes algebra $B_r(G)=C^*_r(G)^*$.

Thus, large quotients of $C^*(G)$ carrying coactions are classified by \emph{large} ideals of $B(G)$.
When we began this study we wondered whether these ideals of $B(G)$ could be used
to classify all large quotients of $B\rtimes_{\alpha} G$ carrying dual coactions;
however
Buss and Echterhoff 
have recently found
a counterexample \cite[Example~5.3]{BusEch}.

Nevertheless it appears that there are lots of these ``exotic ideals'':
it has been attributed to Okayasu~\cite{okayasu} and (independently)
to Higson and Ozawa (see \cite[Remark~4.5]{BrownGuentner})
that for $2\le p<\infty$, the ideals $E_p$ of $B(\F_2)$ formed by taking the weak* closures of $B(\F_2)\cap \ell^p(\F_2)$
are all different.

We use these large ideals $E$ of $B(G)$ to generate intermediate crossed products via slicing:
the dual coaction $\what\alpha$ of $G$
gives a module action of $B(G)$ on $B\rtimes_\alpha G$
by
\[
f\cdot a=(\id\otimes f)\circ \what\alpha(a).
\]
It turns out that the kernel of the regular representation $\Lambda:B\rtimes_\alpha G\to B\rtimes_{\alpha,r} G$
comprises the elements that are killed by $B_r(G)$.
Thus the ideal $B_r(G)\normal B(G)$ allows us to recover the reduced crossed product.
For any large quotient $q:B\rtimes_\alpha G\to (B\rtimes_\alpha G)/I$ carrying a dual coaction,
it is natural to ask whether there exists a large ideal $E\normal B(G)$ such that
\[
\ker q=\{a\in B\rtimes_\alpha G:E\cdot a=\{0\}\};
\]
in any event,
\secref{once} below
shows that
for a large ideal $E\normal B(G)$,
and any coaction $\delta:A\to M(A\otimes C^*(G))$,
the set
\[
\JJ(E)=\JJ_\delta(E):=\{a\in A:E\cdot a=\{0\}\}
\]
is an ideal of $A$ that is \emph{invariant} in the sense that the quotient $A^E:=A/\JJ(E)$ carries a coaction $\delta^E$.
Note that we've replaced the 
dual coaction $(B\rtimes_\alpha G,\what\alpha)$ by an arbitrary coaction $(A,\delta)$.

In this generality, the replacement for the regular representation $\Lambda:B\rtimes_\alpha G\to B\rtimes_{\alpha,r} G$
is the \emph{normalization}
\[
q^n:(A,\delta)\to (A^n,\delta^n),
\]
and we have a commuting diagram
\[
\xymatrix{
A \ar[rr]^-\delta \ar[dd]_{q^n} \ar[dr]^{q^E}
&&A\otimes C^*(G) \ar'[d]^{q^n\otimes\id}[dd] \ar[dr]^{q^E\otimes \id}
\\
&A^E \ar[rr]^(.3){\delta^E} \ar[dl]
&&A^E\otimes C^*(G) \ar[dl]
\\
A^n \ar[rr]_-{\delta^n}
&&A^n\otimes C^*(G)
}
\]
The aforementioned counterexample of \cite{BusEch} shows that not all 
large quotients of $(A,\delta)$ arise
this way;
nevertheless, we feel that this tool deserved to become more widely known.

Actually, 
our original motivation in writing this paper
involves crossed-product duality;
everything we need can be found in, e.g., \cite[Appendix~A]{BE}, \cite{boiler}, and \cite{ekq},
and in the following few sentences we very briefly recall the essential facts.
The Imai-Takai duality theorem and its 
modernization due to Raeburn say that 
if $\alpha$ is an action of a locally compact group $G$ on a $C^*$-algebra $B$,
there
is a dual coaction $\what\alpha$ of $G$ such that 
$B\rtimes_\alpha G\rtimes_{\what\alpha} G\cong B\otimes\KK(L^2(G))$.
Katayama gave a dual version of crossed-product duality, starting with a coaction $\delta$ of $G$ on a $C^*$-algebra $A$:
there is a dual action $\what\delta$ of $G$ on the crossed product $A\rtimes_\delta G$ such that
$A\rtimes_\delta G\rtimes_{\what\delta,r} G\cong A\otimes\KK$.
However, Katayama used what are nowadays called \emph{reduced} coactions --- 
more recently, crossed-product duality has been reworked in terms of Raeburn's \emph{full} coactions,
and the modern version of Katayama's theorem gives the same isomorphism for (full) coactions that are \emph{normal}, i.e., embed faithfully into $A\rtimes_\delta G$.
On the other hand, it is known that for some other coactions, which are called \emph{maximal}, crossed-product duality uses the \emph{full} crossed product by the dual action:
$A\rtimes_\delta G\rtimes_{\what\delta} G\cong A\otimes\KK$.

Thus, noncommutative crossed-product duality has been complicated by the different choices of action-crossed product (i.e., full vs reduced) from the outset.
But the situation is even more complicated: there exist coactions that are neither normal nor maximal,
so that neither the reduced nor the full version of crossed-product duality holds.
This can be sorted out using the \emph{canonical surjection}
$\Phi:A\rtimes_\delta G\rtimes_{\what\delta} G\to A\otimes\KK$,
which is an isomorphism precisely when the coaction $\delta$ is maximal,
and which factors through an isomorphism
$A\rtimes_\delta G\rtimes_{\what\delta,r} G\cong A\otimes\KK$
precisely when $\delta$ is normal.
Every (full) coaction $(A,\delta)$ 
has a \emph{maximalization} and a \emph{normalization},
meaning it
sits in a diagram
$\psi:(A^m,\delta^m)\to (A,\delta)\to (A^n,\delta^n)$
of equivariant surjections,
where the first and third coactions are maximal and normal, respectively, and all three crossed products are isomorphic.
It follows that the kernel of the canonical surjection $\Phi$ is contained in the kernel of the regular representation
$\Lambda:A\rtimes_\delta G\rtimes_{\what\delta} G\to A\rtimes_\delta G\rtimes_{\what\delta,r} G$,
and hence gives a commuting diagram
\[
\xymatrix{
A\rtimes_\delta G\rtimes_{\what\delta} G \ar[r]^-\Phi \ar[d]^Q
&A\otimes\KK
\\
(A\rtimes_\delta G\rtimes_{\what\delta} G)/\ker\Phi, \ar[ur]_\cong
}
\]
where $Q$ is the quotient map.

Thus the coaction $(A,\delta)$ can be regarded to have a ``type'' determined by how the ideal $\ker\Phi$ sits inside $\ker\Lambda$,
with the maximal coactions corresponding to $\ker\Phi=\{0\}$ and the
normal coactions corresponding to $\ker\Phi=\ker\Lambda$.
We would like to have some more intrinsic way to determine what ``type'' $\delta$ has,
namely the kernel of the maximalization map $A^m\to A$.
So, a natural question arises:
if we start with a maximal coaction $(A,\delta)$
is there some way to classify the ideals of $A$ that give rise to coactions intermediate between $\delta$ and the normalization $\delta^n$, and moreover what can we say about these ideals with regard to crossed-product duality?

As indicated above, here we investigate ideals of $A$ determined by ``large'' ideals of $B(G)$, by which we mean weak* closed $G$-invariant ideals of $B(G)$ containing $B_r(G)$.
In~\secref{prelim} we review some preliminaries on coactions. Then in \secref{once} we show how every large ideal $E$ of $B(G)$ determines a coaction $(A^E,\delta^E)$ on a quotient of $A$.
In \secref{duality} we show that
a quotient coaction $(A/J,\delta_J)$ 
of a maximal coaction $(A,\delta)$
is of the form
$(A^E,\delta^E)$ for some large ideal $E$ of $B(G)$
if and only if it
satisfies
a sort of \emph{$E$-crossed-product duality}, involving what we call the \emph{$E$-crossed product} $A\rtimes_\delta G\rtimes_{\what\delta,E} G$.
During the last stage of writing this paper, we learned that Buss and Echterhoff had also proved one direction of this latter result \cite[Theorem~5.1]{BusEch}; our methods are significantly different from theirs.
In the case of the canonical coaction $(C^*(G),\delta_G)$, we show that the above ideals $E$ of $B(G)$ give a complete classification of the quotient coactions $(A,\delta)$ sitting between $(C^*(G),\delta_G)$ and the normalization $(C^*_r(G),\delta_G^n)$.
After the completion of this paper, we learned of a second paper of Buss and Echterhoff
\cite{BusEch2} that is also relevant to this work.

We originally wondered 
whether every coaction 
satisfies $E$-crossed product duality
for some $E$.
In \cite[Conjecture~6.12]{graded} we even conjectured that this would be true for dual coactions.
However,
the counterexample of
Buss and Echterhoff 
\cite[Example~5.3]{BusEch} gives a negative answer.

From \secref{counter} onward we will restrict to the case of 
coactions satisfying a certain ``slice properness'' condition,
which we introduce in \secref{properness}.
We impose this hypothesis
to make the $B(G)$-module action on $A$ appropriately continuous.
%
After we submitted this manuscript, we learned that our \defnref{proper def} of proper coaction 
is a special case of \cite[Definition 2.4]{Ellwood}, 
which concerns actions of Hopf C*-algebras.
Our definition is also closely related to Condition~(A1) in \cite[Section~4.1]{Goswami-Kuku}, which concerns discrete quantum groups and involves the algebraic tensor product.
We are grateful to the referee for drawing these references to our attention.

In \secref{counter} we give examples
of quotient coactions that are not determined by any large ideal $E$ of $B(G)$.
These examples actually turn out to be similar to (and discovered independently from) those in \cite{BusEch},
although they do not do the full job that those of Buss and Echterhoff do, namely they do not involve the maximalization.

In \secref{twice} we start with a maximal coaction $(A,\delta)$ and two large ideals $E_1\supset E_2$ of $B(G)$, and investigate the question of whether the quotient $(A^{E_1},\delta^{E_1})\to (A^{E_2},\delta^{E_2})$ is determined by any third ideal $E$.
In the case of the canonical coaction $(C^*(G),\delta_G)$, we give a list of equivalent conditions,
although the general question is still left open.
Finally, in \secref{lp} we specialize further to the 
study of ideals $E_p$ obtained from $L^p(G)$, where,
although 
we cannot completely answer the question regarding the quotient
$(A^{E_1},\delta^{E_1})\to (A^{E_2},\delta^{E_2})$, we are at least able to learn enough to obtain examples of intermediate quotients between $C^*(G)$ and $C^*_r(G)$ on which $\delta_G$ descends to a comultiplication (not a coaction!) that fails to be injective.

\section{Preliminaries}\label{prelim}

For the definitions and basic facts about coactions of locally compact groups on $C^*$-algebras and imprimitivity bimodules, we refer to \cite{BE}.  
Here we briefly summarize the less-standard concepts and notation we will need.

If $J$ is an ideal (always closed and two-sided) of $A$ and $Q:A\to A/J$ is the quotient map,
we say that $J$ is \emph{$\delta$-invariant} if
\[
J\subset \ker (Q\otimes\id)\circ\delta;
\]
equivalently
(by \cite[Lemma~3.11]{graded} for example),
if $Q$ is $\delta-\delta_J$ equivariant for a unique coaction $\delta_J$ on $A/J$.
All quotient coactions arise in essentially this way.


\begin{lem}\label{morita}
Suppose $(A,\delta)$ and $(B,\epsilon)$ are two coactions of $G$,
$X$ is an $A-B$ imprimitivity bimodule,
$\zeta$ is a $\delta-\epsilon$ compatible coaction of $G$ on $X$,
$K$ is an $\epsilon$-invariant ideal of $B$, and
$J=X{-}\ind K$ is the Rieffel-equivalent ideal of $A$.
Then $J$ is $\delta$-invariant.
\end{lem}

\begin{proof}
$J$ is densely spanned by elements of the form ${}_A\<\xi,\eta\cdot b\>$, where $\xi,\eta\in X$ and $b\in K$.
Let $Q:A\to A/J$ {and} $R:B\to B/K$
be the quotient maps.
We want to show that
\[
(Q\otimes\id)\circ \delta\bigl({}_A\<\xi,\eta\cdot b\>\bigr)=0.
\]
Since $X$ is an $A-B$ imprimitivity bimodule,
$X\otimes C^*(G)$ is an $(A\otimes C^*(G))-(B\otimes C^*(G)))$ imprimitivity bimodule.
The quotient map $S:X\to X/X\cdot K$ is a $Q-R$ compatible imprimitivity bimodule homomorphism,
so
\[
S\otimes \id:(X\otimes C^*(G))\to (X/X\cdot K\otimes C^*(G))
\]
is a $(Q\otimes\id)-(R\otimes\id)$ compatible imprimitivity bimodule homomorphism.
It suffices to show that the multiplier
\[
(Q\otimes\id)\circ \delta\bigl({}_A\<\xi,\eta\cdot b\>\bigr)\in M(A/J\otimes C^*(G))
\]
kills every element of the module $X/X\cdot K\otimes C^*(G)$, and we can take this arbitrary element to be of the form $(S\otimes\id)(\kappa)$, where $\kappa\in X\otimes C^*(G)$.
We compute:
\begin{align*}
&(Q\otimes\id)\circ \delta\bigl({}_A\<\xi,\eta\cdot b\>\bigr)\cdot (S\otimes\id)(\kappa)
\\&\quad=(S\otimes\id)\Bigl({}_{M(A\otimes C^*(G))}\bigl\<\zeta(\xi),\zeta(\eta\cdot b)\bigr\>\cdot \kappa\Bigr)
\\&\quad=(S\otimes\id)\Bigl(\zeta(\xi)\cdot\bigl\<\zeta(\eta)\cdot \epsilon(b),\kappa\bigr\>_{M(B\otimes C^*(G))}\Bigr)
\\&\quad=(S\otimes\id)\Bigl(\zeta(\xi)\cdot\epsilon(b)^*\bigl\<\zeta(\eta),\kappa\bigr\>_{M(B\otimes C^*(G))}\Bigr)
\\&\quad=(S\otimes\id)\circ\zeta(\xi)\cdot(R\otimes\id)\Bigl(\epsilon(b)^*\bigl\<\zeta(\eta),\kappa\bigr\>_{M(B\otimes C^*(G))}\Bigr)
\\&\quad=(S\otimes\id)\circ\zeta(\xi)\cdot(R\otimes\id)\circ\epsilon(b)^*(R\otimes\id)\Bigl(\bigl\<\zeta(\eta),\kappa\bigr\>_{M(B\otimes C^*(G))}\Bigr)
\\&\quad=0,
\end{align*}
since $b\in \ker(R\otimes\id)\circ\epsilon$.
\end{proof}

Adapting
the definition from 
\cite[Definition~2.7]{lprs}, 
where it appears for reduced coactions, we 
say that a  unitary $U$ in $M(A\otimes C^*(G))$ is a \emph{cocycle}
for a coaction $(A,\delta)$ if 
\begin{enumerate}
\item $\id\otimes\delta_G(U) = (U\otimes 1)(\delta\otimes\id(U))$, and
\item $U\delta(A)U^*(1\otimes C^*(G))\subset A\otimes C^*(G)$.
\end{enumerate}

Note that (ii) above implies
\[
(1\otimes C^*(G))U\delta(A)U^*\subset A\otimes C^*(G).
\]

It is mentioned in \cite{lprs} that
in this case $\ad U\circ\delta$ is also a coaction,
which is said to be \emph{exterior equivalent} to $\delta$.
However, there is a disconnect here: in \cite{lprs}, the definition of coaction on a $C^*$-algebra did not include the nondegeneracy condition
\begin{equation}\label{coaction nondegenerate}
\clspn\{\delta(A)(1\otimes C^*(G)\}=A\otimes C^*(G),
\end{equation}
whereas nowadays this condition is built into the definition of coaction.
Thus (modulo the passage from \emph{reduced} to \emph{full} coactions --- see \cite{boiler}), $\epsilon=\ad U\circ\delta$ satisfies all the conditions in the definition of coaction except, ostensibly, nondegeneracy.
In \cite[Paragraph preceding Lemma~2.6]{ekq},
it is stated that nondegeneracy of $\epsilon$ follows from that of $\delta$,
and the justification is that exterior equivalent coactions are Morita equivalent,
and \cite[Proposition~2.3]{kq:imprimitivity} shows that Morita equivalence of $C^*$-coactions preserves nondegeneracy.
Somehow irritating, the observation that  exterior equivalence
implies Morita equivalence for coactions seems not to be readily available in the literature, so for completeness we record the details here.

\begin{prop}\label{exterior}
Let $U$ be a cocycle for a coaction $\delta$ of $G$ on $A$,
and let $\epsilon=\ad U\circ\delta$ be the associated exterior equivalent coaction.
Let $X$ be the standard $A-A$ imprimitivity bimodule,
and define $\zeta:X\to M(X\otimes C^*(G))$ by
\[
\zeta(x)=U\delta(x)\righttext{for}x\in X=A.
\]
Then $\zeta$ is an $\epsilon-\delta$ compatible coaction.
\end{prop}

\begin{proof}
First of all, it is clear that
\[
\zeta(X)\subset M(X\otimes C^*(G))=\LL_{A\otimes C^*(G)}(A\otimes C^*(G),X\otimes C^*(G)).
\]
For $a\in A$ and $x,y\in X$ we have
\begin{align*}
\zeta(a\cdot x)
&=U\delta(ax)
=U\delta(a)\delta(x)
=\epsilon(a)U\delta(x)
=\epsilon(a)\cdot \zeta(x),
\end{align*}
and
\begin{align*}
\<\zeta(x),\zeta(y)\>_{A\otimes C^*(G)}
&=(U\delta(x))^*(U\delta(y))
=\delta(x^*)U^*U\delta(y)
=\delta(x^*y).
\end{align*}
By \cite[Definition~1.14 and Remark~1.17 (2)]{BE},
it now follows that $\epsilon$ is a
\emph{possibly degenerate} coaction.  
But since $\delta$ does satisfy \eqref{coaction nondegenerate}
by assumption, we can safely appeal to \cite[Proposition~2.3]{kq:imprimitivity} to conclude that $\epsilon$ is also nondegenerate.
\end{proof}


\begin{rem}\label{double dual}
It follows from \cite[Lemma~3.8 and its proof]{ekq} that
if we define $W=(M\otimes\id)(w_G)\in M(\KK(L^2(G))\otimes C^*(G))$,
and let
$\delta\otimes_*\id$ denote the coaction $(\id\otimes\Sigma)\circ(\delta\otimes\id)$,
where $\Sigma$ is the flip map on $C^*(G)\otimes C^*(G)$,
then $1\otimes W^*$ is a cocycle for $\delta\otimes_*\id$, and 
the canonical surjection
$\Phi:A\rtimes_\delta G\rtimes_{\what\delta} G\to A\otimes\KK(L^2(G))$ is
$\what{\what\delta}-\ad (1\otimes W^*)\circ(\delta\otimes_*\id)$ equivariant.
\end{rem}

There are several choices for the conventions regarding a \emph{Galois correspondence} between partially ordered sets $X$ and $Y$ ---
we will take this to mean a pair of order-reversing functions $f:X\to Y$ and $g:Y\to X$ such that
\[
\id_X\le g\circ f
\midtext{and}
\id_Y\le f\circ g.
\]
These properties have the following well-known consequences:
\begin{enumerate}
\item $f\circ g\circ f=f$ and $g\circ f\circ g=g$;
\item $f(x)\ge y$ if and only if $x\le g(y)$;
\item $g\circ f(x)=g\circ f(x')\Rightarrow f(x)=f(x')$;
\item $f\circ g(y)=f\circ g(y')\Rightarrow g(y)=g(y')$.
\end{enumerate}

\section{$E$-determined coactions}\label{once}

In this section we show how certain ideals of $B(G)$ produce quotients of coactions, although we will begin with 
quite general subsets of $B(G)$.

We recall some notation and results from \cite{graded}.
For any weak*-closed subspace $E\subset B(G)$,
the preannihilator $\pann E$ in $C^*(G)$ is a (closed, two-sided) ideal
if and only if $E$ is invariant under the $G$-bimodule action,
if and only if $E$ is invariant under the $C^*(G)$-bimodule action.
Write $C^*_E(G)=C^*(G)/\pann E$, and let $q_E:C^*(G)\to C^*_E(G)$ be the quotient map.
The dual map $q_E^*:C^*_E(G)^*\to B(G)$ is an isometric isomorphism onto $E$, and we identify $E$ with $C^*_E(G)^*$ and regard $q_E^*$ as the inclusion map.
The canonical coaction $\delta_G$ on $C^*(G)$ descends to a coaction $\delta_G^E$ on $C^*_E(G)$ if and only if $E$ is an ideal of $B(G)$.

\begin{defn}
We call an ideal of $B(G)$ \emph{large} if it is weak* closed, $G$-invariant, and contains $B_r(G)$;
by \cite[Lemma~3.14]{graded} the latter containment condition is satisfied as long as the ideal is nonzero.
\end{defn}

\begin{defn}\label{J}
Let $(A,\delta)$ be a coaction.
For any weak* closed subspace $E\subset B(G)$, define
\[
\JJ(E)=\JJ_\delta(E)=\{a\in A:f\cdot a=0\text{ for all }f\in E\}.
\]
\end{defn}

\begin{thm}
For any weak* closed $G$-invariant subspace $E$ of $B(G)$,
\[
\JJ(E)=\ker(\id\otimes q_E)\circ\delta.
\]
\end{thm}

\begin{proof}
We can identify $E$ with $C^*_E(G)^*$, and the dual map $q_E^*:C^*_E(G)^*\to C^*(G)^*$ with the inclusion map $E\hookrightarrow B(G)$.
Since the slice maps $\id\otimes f$ for $f\in E$ separate the points of $A\otimes C^*_E(G)$, if $a\in A$ then
$a\in \ker(\id\otimes q_E)\circ\delta$ if and only if for all $f\in E$ we have
\begin{align*}
f\cdot a
=(\id\otimes f)\circ\delta(a)
&=(\id\otimes q_E^*)(f)\circ\delta(a)\\
&=(\id\otimes f)\circ (\id\otimes q_E)\circ\delta(a)
=0;
\end{align*}
i.e., if and only if $a\in \JJ(E)$.
\end{proof}

\begin{cor}
For every weak* closed $G$-invariant subspace $E$ of $B(G)$, $\JJ(E)$ is an ideal of $A$.
\end{cor}

\begin{lem}\label{invariant}
For every coaction $(A,\delta)$ and every weak* closed $G$-invariant ideal of $E$ of $B(G)$, the ideal $\JJ(E)$ of $A$ is $\delta$-invariant.
\end{lem}

\begin{proof}
We first show that $\JJ(E)$ is a $B(G)$-submodule: if $a\in \JJ(E)$, $f\in B(G)$, and $g\in E$ then
\[
g\cdot (f\cdot a)=(gf)\cdot a=0
\]
because $gf\in E$ as $E$ is an ideal. Thus $f\cdot a\in\JJ(E)$.

Let $Q:A\to A/\JJ(E)$ be the quotient map.
We must show that if $a\in \ker Q=\JJ(E)$ then $(Q\otimes\id)\circ\delta(a)=0$,
and it suffices to observe that for all $\omega\in (A/\JJ(E))^*$ and $f\in B(G)$ we have
\begin{align*}
(\omega\otimes f)\circ(Q\otimes\id)\circ\delta(a)
&=(Q^*\omega\otimes f)\circ\delta(a)
=Q^*\omega(f\cdot a)
=0,
\end{align*}
because $Q^*\omega\in \JJ(E)\ann$ and $f\cdot a\in \JJ(E)$.
\end{proof}

\begin{notn}
For a weak* closed $G$-invariant ideal $E$ of $B(G)$, let $A^E=A/\JJ(E)$, and let $\delta^E$ be the associated quotient coaction on $A^E$, whose existence is assured by \lemref{invariant} and \cite[Lemma~3.11]{graded}.
\end{notn}

We are quite interested in coactions that arise in this way;
slightly more generally, we are interested in 
equivariant surjections $\varphi:A\to B$ for which $\ker\varphi=\JJ(E)$, so that 
there an isomorphism $\theta$ making the diagram
\[
\xymatrix{
(A,\delta) \ar[d]_Q \ar[dr]^\varphi
\\
(A^E,\delta^E) \ar[r]_-\theta^-\cong
&(B,\epsilon)
}
\]
commute, where $Q$ is the quotient map.

\begin{defn}\label{E-determined}
For a large ideal $E$ of $B(G)$
and an equivariant surjection $\varphi:(A,\delta)\to (B,\epsilon)$,
we say $(B,\epsilon)$
is
\emph{$E$-determined from $(A,\delta)$},
or just \emph{$E$-determined} when $(A,\delta)$ is understood,
if
$\ker\varphi=\JJ_\delta(E)$.
\end{defn}

\begin{ex}
Standard coaction theory guarantees that the normalization $(A^n,\delta^n)$ is $B_r(G)$-determined from $(A,\delta)$, and $(A,\delta)$ is $B(G)$-determined from itself, because $q_{B(G)}$ is the identity map.
\end{ex}

\thmref{counterexample}
gives  examples showing that not every quotient of a coaction $(A,\delta)$ is necessarily $E$-determined by some large ideal $E$ of $B(G)$.
\cite[Example~5.4]{BusEch} gives examples where the coaction $(A,\delta)$ is maximal.

\begin{defn}
Let $(A,\delta)$ be a coaction.
A $\delta$-invariant ideal of $A$ is \emph{small} if it is contained in $\ker j_A$, and
a quotient $(B,\epsilon)$ of $(A,\delta)$ is \emph{large} if the kernel of the quotient map $A\to B$ is small.
\end{defn}

\begin{obs}
Let $(A,\delta)$ be a coaction, and let $E$ be a large ideal of $B(G)$.
Then $\JJ(E)$ is small.
\end{obs}

\begin{rem}
Note that every coaction $(A,\delta)$ is a large quotient of its maximalization $(A^m,\delta^m)$.
Also, the small ideals of $C^*(G)$ are precisely the preannihilators of the large ideals of $B(G)$.
\end{rem}

\section{$E$-crossed product duality}\label{duality}

Let $(A,\delta)$ be a coaction, and let
\[
\Phi:A\rtimes_\delta G\rtimes_{\what\delta} G\to A\otimes \KK
\]
be the canonical surjection, where $\KK=\KK(L^2(G))$.

\begin{lem}\label{K}
The ideal $\ker\Phi$ is 
small.
\end{lem}

\begin{proof}
By \cite[Lemmas~3.6 and 3.8]{ekq} the surjection $\Phi$ is 
equivariant for two coactions, where the coaction on
$A\rtimes_\delta G\rtimes_{\what\delta} G$, denoted by $\wilde\delta$ in \cite{ekq},
is exterior equivalent, and hence Morita equivalent, to the double-dual coaction $\what{\what\delta}$.
Since $\Phi$ transports $\wilde\delta$ to \emph{some} coaction on $A\otimes\KK$, by \cite[Lemma~3.11]{graded} the ideal $\ker\Phi$ is $\widetilde\delta$-invariant.
So, by \lemref{morita}, $\ker\Phi$ is also $\what{\what\delta}$-invariant.

For the other part, 
by \cite[Proposition~2.2]{ekq} there is a surjection $\Psi$ making the diagram
\[
\xymatrix{
A\rtimes_\delta G\rtimes_{\what\delta} G \ar[r]^-\Phi \ar[d]_\Lambda
&A\otimes \KK(L^2(G)) \ar[dl]^\Psi
\\
A\rtimes_\delta G\rtimes_{\what\delta,r} G
}
\]
commute, 
where
\[
\Lambda=\Lambda_\delta:A\rtimes_\delta G\rtimes_{\what\delta} G\to A\rtimes_\delta G\rtimes_{\what\delta,r} G
\]
is the regular representation.
Thus
$\ker\Phi$ is small,
since $A\rtimes_\delta G\rtimes_{\what\delta,r} G$ is the normalization of
$A\rtimes_\delta G\rtimes_{\what\delta} G$.
\end{proof}

\begin{ex}
By \lemref{K}, the extremes for the ideal $\ker\Phi$ are:
\begin{enumerate}
\item $\delta$ is maximal if and only if $\ker\Phi=\{0\}$.
\item $\delta$ is normal if and only if $\ker\Phi=\ker\Lambda$.
\end{enumerate}
\end{ex}

\begin{defn}
A coaction $(A,\delta)$ satisfies \emph{$E$-crossed product duality} if
\[
\ker \Phi=\JJ_{\what{\what\delta}}(E).
\]
\end{defn}

\begin{rem}
This is called ``$E$-duality'' in \cite{BusEch}.
\end{rem}

Thus, $(A,\delta)$ satisfies $E$-crossed product duality exactly when there is an isomorphism $\Psi$ making the diagram
\[
\xymatrix{
A\rtimes G\rtimes G \ar[r]^-{\Phi} \ar[d]_{Q_{\what{\what\delta},E}}
&A\otimes\KK
\\
(A\rtimes G\rtimes G)^E \ar@{-->}[ur]_\Psi^\cong
}
\]
commute, where
\[
(A\rtimes G\rtimes G)^E=(A\rtimes G\rtimes G)/\JJ_{\what{\what\delta}}(E)
\]
and $Q_{\what{\what\delta},E}$ is the quotient map.

\begin{ex}
$(A,\delta)$ is maximal if and only if it satisfies $B(G)$-crossed product duality, and 
normal if and only if it satisfies $B_r(G)$-crossed product duality.
\end{ex}

Now, $(A,\delta)$ is a large quotient of 
its maximalization $(A^m,\delta^m)$; let $\psi:A^m\to A$ be the associated $\delta^m-\delta$ equivariant surjection.
Recall that, if $E$ is a large ideal of $B(G)$,
we say that $(A,\delta)$ is $E$-determined from its maximalization if
$\ker\psi=\JJ_{\delta^m}(E)$.

The following theorem shows that the above two properties on $(A,\delta)$ are equivalent.
In the final stage of writing this paper we learned of the paper \cite{BusEch} by Buss and Echterhoff, and their Theorem~5.1 gives a proof of the converse direction using significantly different techniques.

\begin{thm}\label{equivalent}
$(A,\delta)$ satisfies $E$-crossed product duality if and only if it is $E$-determined from its maximalization.
\end{thm}

\begin{proof}
We must show that
\[
\ker\psi=\JJ_{\delta^m}(E)
\]
if and only if
\[
\ker\Phi=\JJ_{\what{\what\delta}}(E).
\]

Since $(A^m,\delta^m)$ is maximal, the canonical surjection
\[
\Phi_m:A^m\rtimes G\rtimes G\to A^m\otimes\KK
\]
is an isomorphism.
Since $(A,\delta)$ is a large quotient of $(A^m,\delta^m)$, the double crossed product map
\[
\psi\times G\times G:A^m\rtimes G\rtimes G\to A\rtimes G\rtimes G
\]
is an isomorphism, by the folklore \lemref{iso} below.
By functoriality of the constructions, the diagram
\[
\xymatrix{
A^m\rtimes G\rtimes G \ar[r]^-{\Phi_m}_-\cong \ar[d]_{\psi\times G\times G}^\cong
&A^m\otimes\KK \ar[d]^{\psi\otimes\id}
\\
A\rtimes G\rtimes G \ar[r]_{\Phi}
&A\otimes\KK
}
\]
commutes.
Thus,
\[
\Phi_m\circ (\psi\times G\times G)\inv(\ker\Phi)=\ker\psi\otimes\KK.
\]

Our strategy is to show that
\begin{equation}\label{JJ}
\Phi_m\circ (\psi\times G\times G)\inv(\JJ_{\what{\what\delta}}(E))=\JJ_{\delta^m}(E)\otimes\KK.
\end{equation}
Since $\Phi_m\circ (\psi\times G\times G)\inv$ is an isomorphism,
and for ideals $I,J$ of $A^m$ we have $I\otimes\KK=J\otimes\KK$ if and only if $I=J$,
this will suffice.
Since $\psi\times G\times G$ is a $\what{\what{\delta^m}}-\what{\what\delta}$ equivariant isomorphism,
\[
\psi\times G\times G(\JJ_{\what{\what{\delta^m}}}(E))=\JJ_{\what{\what\delta}}(E).
\]
Thus, it suffices to show
\begin{equation}
\Phi_m(\JJ_{\what{\what{\delta^m}}}(E))=\JJ_{\delta^m}(E)\otimes\KK.
\label{Phi E}
\end{equation}
Here are the steps:
\begin{align}
\Phi_m(\JJ_{\what{\what{\delta^m}}}(E))
&=\JJ_{\ad (1\otimes W^*)\circ(\delta^m\otimes_*\id)}(E)\label{doubledual}
\\&=\JJ_{\delta^m\otimes_*\id}(E)\label{cocycle}
\\&=\JJ_{\delta^m}(E)\otimes\KK.\label{tensor}
\end{align}
\eqref{doubledual} follows from $\what{\what{\delta^m}}-\ad(1\otimes W^*)\circ(\delta^m\otimes_*\id)$ equivariance of $\Phi_m$,
\eqref{cocycle} follows because $1\otimes W^*$ is a $\delta^m\otimes_*\id$-cocycle 
(as in \remref{double dual})
---
see the elementary \lemref{abstract cocycle} below ---
and
\eqref{tensor} follows from a routine computation with tensor products:
\begin{align*}
\JJ_{\delta^m\otimes_*\id}(E)
&=\ker \bigl((\id\otimes\id\otimes q_E)\circ (\delta^m\otimes_*\id)\bigr)
\\&=\ker \bigl((\id\otimes\id\otimes q_E)\circ (\id\otimes\Sigma)\circ (\delta^m\otimes\id)\bigr)
\\&=\ker \bigl((\id\otimes\Sigma)\circ (\id\otimes q_E\otimes\id)\circ (\delta^m\otimes\id)\bigr)
\\&=\ker \bigl((\id\otimes q_E\otimes\id)\circ (\delta^m\otimes\id)\bigr)
\quad\text{(since $\id\otimes\Sigma$ is injective)}
\\&=\ker \Bigl(\bigl((\id\otimes q_E)\circ\delta^m\bigr)\otimes\id\Bigr)
\\&=\ker \bigl((\id\otimes q_E)\circ\delta^m\bigr)\otimes\KK
\quad\text{(since $\KK$ is exact)}
\\&=\JJ_{\delta^m}(E)\otimes\KK.
\qedhere
\end{align*}
\end{proof}

In the above proof we invoked the following two general lemmas.
The first relies upon the fact that the normalization map $A\to A^n$ gives isomorphic crossed products $A\rtimes_\delta G\cong A^n\rtimes_{\delta^n} G$,
while the second shows that exterior equivalent coactions have the same $\JJ$ map from large ideals of $B(G)$ to small ideals of $A$.

\begin{lem}\label{iso}
Let $(A,\delta)$ be a coaction,
let $J$ be an invariant ideal,
let $Q:A\to A/J$ be the quotient map,
and let $\delta_J$ be the associated coaction on $A/J$.
Then $J$ is small if and only if the crossed-product homomorphism
\[
Q\times G:A\rtimes_\delta G\to A/J\rtimes_{\delta_J} G
\]
is an isomorphism.
\end{lem}

\begin{proof}
$Q\times G$ is always a surjection, so the issue is whether it is injective.
First suppose $J$ is small.
Then there is a unique surjection $\zeta$ making the diagram
\[
\xymatrix{
A \ar[r]^-Q \ar[dr]_{j_A}
&A/J \ar@{-->}[d]^\zeta
\\
&j_A(A)
}
\]
commute, and moreover $\zeta$ is $\delta_J-\ad j_G$ equivariant,
where $\ad j_G$ is the inner coaction on $j_A(A)$ implemented by the canonical homomorphism $j_G:C_0(G)\to M(A\times_\delta G)$.
Thus we have
\[
j_A\times G=(\zeta\times G)\circ (Q\times G),
\]
which is injective, and hence $Q\times G$ is injective.

For the other direction, note that
\[
(Q\times G)\circ j_A=j_{A/J}\circ Q,
\]
so, assuming $Q\times G$ is injective, we have
\[
J=\ker Q\subset \ker j_A.
\qedhere
\]
\end{proof}

\begin{lem}\label{abstract cocycle}
Let $(A,\delta)$ be a coaction, $U$ a $\delta$-cocycle, and $E$ a large ideal of $B(G)$.
Then
\[
\JJ_{\delta}(E)=\JJ_{\ad U\circ\delta}(E).
\]
\end{lem}

\begin{proof}
We have
\begin{align*}
\JJ_{\ad U\circ\delta}(E)
&=\ker (\id\otimes q_E)\circ\ad U\circ\delta
\\&=\ker \bigl(\ad (\id\otimes q_E)(U)\bigr)\circ (\id\otimes q_E)\circ\delta
\\&=\ker (\id\otimes q_E)\circ\delta
\quad\text{(since $(\id\otimes q_E)(U)$ is unitary)}
\\&=\JJ_{\delta}(E).
\qedhere
\end{align*}
\end{proof}

We can now settle \cite[Conjecture~6.14]{graded} affirmatively (again, see \cite[Theorem~5.1]{BusEch} for an alternative proof):

\begin{cor}
For any large ideal $E$ of $B(G)$,
the coaction $(C^*_E(G),\delta_G^E)$ satisfies $E$-crossed product duality,
and more generally so does the dual coaction of $G$ on an $E$-crossed product
$B\rtimes_{\alpha,E} G$ for any action $(B,G,\alpha)$.
\end{cor}

\section{Slice proper coactions}\label{properness}

\begin{defn}\label{proper def}
A coaction $(A,\delta)$ is 
\emph{proper} if
\begin{equation}\label{proper coaction}
(A\otimes 1)\delta(A)\subset A\otimes C^*(G),
\end{equation}
and 
\emph{slice proper} if
\begin{equation}\label{slice proper coaction}
(\omega\otimes\id)\circ\delta(A)\subset C^*(G)\righttext{for all}\omega\in A^*.
\end{equation}
\end{defn}

Note that 
proper coactions are always slice proper, 
since by the Cohen-Hewitt factorization theorem every functional in $A^*$ can be expressed in the form $\omega\cdot a$, where
\[
\omega\cdot a(b)=\omega(ab)\righttext{for}\omega\in A^*\text{ and }a,b\in A.
\]
On the other hand, elementary examples
show that a coaction can be slice proper without being proper.

Just as every action of a compact group is proper (in the classical sense), every coaction of a discrete group is proper, because then we in fact have $\delta(A)\subset A\otimes C^*(G)$.
In this paper we will only require the weaker notion of slice properness.  
We intend to study proper coactions more thoroughly in upcoming work.

%

Our primary interest in slice proper coactions is the following weak* continuity property:

\begin{lem}\label{weak*}
A coaction
$(A,\delta)$ is slice proper if and only if for all $a\in A$ the map $f\mapsto f\cdot a$ is continuous from the weak* topology of $B(G)$ to the weak topology of $A$.
\end{lem}

\begin{proof}
First assume that $\delta$ is slice proper.
Let $f_i\to 0$ weak* in $B(G)$.
We must show that $f_i\cdot a\to 0$ weakly in $A$, so let $\omega\in A^*$, and compute:
\begin{align*}
\omega(f_i\cdot a)
&=\omega\bigl((\id\otimes f_i)\circ \delta(a)\bigr)
=f_i\bigl((\omega\otimes\id)\circ \delta(a)\bigr)
\to 0
\end{align*}
because $(\omega\otimes\id)\circ \delta(a)\in C^*(G)$ by hypothesis.

Conversely, if $f\mapsto f\cdot a$ is weak* to weakly continuous
and $f_i\to 0$ weak* in $B(G)$,
then for all $\omega\in A^*$ we have
\[
f_i\bigl((\omega\otimes\id)\circ\delta(a)\bigr)
=\omega(f_i\cdot a)
\to 0,
\]
and so $(\omega\otimes\id)\circ\delta(a)\in C^*(G)$.
\end{proof}

The next result shows that slice properness is preserved by morphisms:

%

\begin{prop}\label{strict morphism}
Let $\phi:A\to M(B)$ be a nondegenerate homomorphism that is equivariant for coactions $\delta$ and $\epsilon$, respectively.
If $\delta$ is 
slice proper, 
then $\epsilon$ is also slice proper.
\end{prop}

\begin{proof}
Let $b\in B$.
We must show that $(\omega\otimes\id)\circ\epsilon(b)\in C^*(G)$ for all $\omega\in B^*$, and it suffices to do it for positive $\omega$.
We have
\[
(\omega\otimes\id)\circ\epsilon(b)\in M(C^*(G)),
\]
so it suffices to show that for every $\psi\in M(C^*(G))^*$ that is in the annihilator of $C^*(G)$ we have
\[
0=\psi\bigl((\omega\otimes\id)\circ\epsilon(b)\bigr)
=(\omega\otimes\psi)\bigl(\epsilon(b)\bigr).
\]
Again, it suffices to do this for positive $\psi$.
Since $\phi$ is nondegenerate we can factor $b=\phi(a^*)c$ with $a\in A$ and $c\in B$.
By the Cauchy-Schwarz inequality for positive functionals on $C^*$-algebras, we have
\begin{align*}
\bigl|(\omega\otimes\psi)\circ\epsilon(b)\bigr|^2
&=\bigl|(\omega\otimes\psi)\circ\epsilon(\phi(a^*)c)\bigr|^2
\\&=\Bigl|(\omega\otimes\psi)\bigl((\phi\otimes\id)\circ\delta(a)^*\epsilon(c)\bigr)\Bigr|^2
\\&\le (\omega\otimes\psi)\bigl((\phi\otimes\id)\circ\delta(a^*a)\bigr)(\omega\otimes\psi)\bigl(\epsilon(c^*c)\bigr)
\\&=\psi\bigl((\phi^*(\omega)\otimes\id)\circ\delta(a^*a)\bigr)(\omega\otimes\psi)\bigl(\epsilon(c^*c)\bigr)
=0
\end{align*}
because
$(\phi^*(\omega)\otimes\id)\circ\delta(a^*a)\in C^*(G)$.
\end{proof}

\section{Counterexamples}\label{counter}

In \cite[Example~5.4]{BusEch}, Buss and Echterhoff give examples of coactions that are not $E$-determined from their maximalizations for any large ideal $E$ of $B(G)$. In \thmref{counterexample} we give related, but different, examples, involving  quotients of not-necessarily maximal coactions.

\begin{defn}
Let $(A,\delta)$ be a slice proper coaction. For any 
small ideal $J$ of $A$ define
\[
\EE(J)=\EE_\delta(J)=\{f\in B(G):(x\cdot f\cdot y)\cdot J=\{0\}\text{ for all }x,y\in G\}.
\]
\end{defn}

\begin{rem}
When $\delta$ is the dual coaction $\what\alpha$ on an action crossed product $B\rtimes_\alpha G$,
we have a simpler definition:
\[
\EE(J)=\{f\in B(G):f\cdot J=\{0\}\},
\]
since the right-hand side is automatically $G$-invariant in this case:
for $x\in G$, $a\in J$, and $f\in B(G)$, if $f\cdot a=0$ then
\begin{align*}
(x\cdot f)\cdot a
&=(\id\otimes x\cdot f)\bigl(\what\alpha(a)\bigr)
\\&=(\id\otimes f)\bigl(\what\alpha(a)(1\otimes x)\bigr)
\\&=(\id\otimes f)\bigl(\what\alpha(a)(i_G(x)\otimes x)\bigr)i_G(x)\inv
\\&=(\id\otimes f)\bigl(\what\alpha(ai_G(x))\bigr)i_G(x)\inv
=0,
\end{align*}
because $J$ is an ideal of $B\rtimes_\alpha G$ and hence is an ideal of $M(B\rtimes_\alpha G)$.
This shows left $G$-invariance, and similarly for right invariance.
Note that we could have shown invariance under slightly weaker hypotheses
on the coaction $(A,\delta)$: 
it suffices to have,
for every $x\in G$,
a unitary element $u_x\in M(A)$ such that $\delta(u_x)=u_x\otimes x$,
or, for another sufficient condition,
when $G$ is discrete
it is enough that
the coaction $(A,\delta)$ be determined by a saturated Fell bundle $\AA\to G$,
i.e., $A$ is the closed span of the fibres $\{A_x\}_{x\in G}$ of the bundle,
$\clspn\{A_xA_x^*\}=A_e$ for all $x\in G$, and
$\delta(a_x)=a_x\otimes x$ for all $a_x\in A_x$.
\end{rem}

\begin{q}
For a slice proper coaction $(A,\delta)$ and a 
small ideal $J$ of $A$,
is the set
\[
\{f\in B(G):f\cdot J=\{0\}\}
\]
$G$-invariant in $B(G)$? Presumably not, but we do not know of a counterexample.
\end{q}

\begin{lem}\label{galois ideal}
For any slice proper coaction $(A,\delta)$,
$\JJ_\delta$ and $\EE_\delta$ form a Galois correspondence
between the large ideals of $B(G)$ and
the 
small ideals of $A$.
\end{lem}

\begin{proof}
We already know that if $E$ is a large ideal of $B(G)$ then
$\JJ(E)$ is 
a
small ideal of $A$,
so it suffices to show that
if $J$ is 
a
small ideal of $A$ then
$\EE(J)$ is a nonzero weak*-closed $G$-invariant ideal of $B(G)$, because
it is obvious that $\JJ$ and $\EE$ are inclusion-reversing,
$\EE(\JJ(E))\supset E$, and $\JJ(\EE(J))\supset J$.
$\EE(J)$ is obviously an ideal of $B(G)$, and it is 
$G$-invariant by definition.
Since the coaction $(A,\delta)$ is slice proper, for every $a\in A$ the map $f\mapsto f\cdot a$ is weak*-weakly continuous by \lemref{weak*}, so $\EE(J)$ is weak* closed.
Since $J\subset \ker j_A$ we have
\[
\EE(J)\supset \EE(\ker j_A)\supset B_r(G),
\]
so $\EE(J)$ is nonzero.
\end{proof}

\begin{ex}
In the case of the coaction $(C^*(G),\delta_G)$, we have:
\begin{itemize}
\item $\JJ(E)=\pann E$;
\item $\EE(J)=J\ann$;
\item $\EE(\JJ(E))=E$;
\item $\JJ(\EE(J))=J$.
\end{itemize}
\end{ex}

\begin{cor}\label{J=J(E)}
Let $(A,\delta)$ be a slice proper coaction,
let $J$ be 
a
small ideal of $A$,
and let $E$ be a large ideal of $B(G)$.
Suppose that $\EE(J)=\EE(\JJ(E))$,
and that $J=\JJ(E')$ for some large ideal $E'$.
Then $J=\JJ(E)$.
\end{cor}

\begin{proof}
This follows from the properties of Galois correspondences.
\end{proof}

\begin{lem}\label{equivariant}
Let $(A,\delta)$ and $(C,\epsilon)$ be slice proper coactions of $G$, 
let $\varphi:A\to M(C)$ be a $\delta-\epsilon$ equivariant nondegenerate homomorphism,
let $J$ be a 
small ideal of $A$,
and let $E$ be a large ideal of $B(G)$.
Then 
\begin{enumerate}
\item
The ideal
\[
\varphi_*(J):=\clspn\{C\varphi(J)C\}
\]
of $C$ is 
small.

\item
$\varphi_*(\JJ_\delta(E))\subset \JJ_\epsilon(E)$.

\item
Suppose
\begin{itemize}
\item $\varphi$ is faithful,

\item $\EE(\JJ_\delta(E))=E$,

\item $C=\clspn\{D\varphi(A)\}$
for a nondegenerate $C^*$-subalgebra $D$ of $M(C)$ such that
$\bar\epsilon(d)=d\otimes 1$ for all $d\in D$, and

\item $\varphi_*(\JJ_\delta(E))=\JJ_\epsilon(E')$ for some $E'$.
\end{itemize}
Then $\varphi_*(\JJ_\delta(E))=\JJ_\epsilon(E)$.
\end{enumerate}
\end{lem}

\begin{rems}
(1)
Note that (iii) above does not say that $E'=E$, even when both are large ideals of $B(G)$.
The hypotheses in (iii) might seem artificial, but we will see several naturally-occuring situations where they are all satisfied.

(2)
Item (ii) above can be used to show that the assignment $(A,\delta)\mapsto (A^E,\delta^E)$ can be parlayed into a functor, as in \cite[Section~6]{BusEch}, but we have no need for this in the current paper.
\end{rems}

\begin{proof}
(i)
Let $Q:A\to A/J$ and $R:C\to C/\varphi_*(J)$ be the quotient maps.
The hypotheses imply that $J\subset \ker \bar R\circ\varphi$, so there is a homomorphism $\psi$ making the diagram
\[
\xymatrix{
A \ar[r]^-\varphi \ar[d]_Q
&M(C) \ar[d]^{\bar R}
\\
A/J \ar@{-->}[r]_-\psi
&M(C/\varphi_*(J))
}
\]
commute.

We must show that
$\varphi_*(J)\subset \ker(R\otimes\id)\circ\epsilon$,
and it suffices to show that
$J\subset \ker (R\otimes\id)\circ\bar\epsilon\circ\varphi$:
for $j\in J$ we have
\begin{align*}
(R\otimes\id)\circ\bar\epsilon\circ\varphi(j)
&=(R\otimes\id)\circ(\varphi\otimes\id)\circ\delta(j)
\\&=\bigl(\bar R\circ \varphi\otimes\id\bigr)\circ\delta(j)
\\&=(\psi\circ Q\otimes\id)\circ\delta(j)
\\&=(\psi\otimes\id)\circ (Q\otimes\id)\circ\delta(j)
\\&=0,
\end{align*}
because $J\subset \ker (Q\otimes\id)\circ\delta$.

To see that $\varphi_*(J)$ is small,
we have
\[
J\subset \ker j_A\subset \ker (\varphi\times G)\circ j_A=\ker (j_C)\circ\varphi,
\]
and it follows that
\[
\varphi_*(J)\subset \ker j_C.
\]

(ii)
If $a\in \JJ_\delta(E)$, then for all $b,c\in C$ we have
\begin{align*}
&(\id\otimes q_E)\circ\epsilon(b\varphi(a)c)
\\&\quad=(\id\otimes q_E)\circ\epsilon(b)(\id\otimes q_E)\circ\bar\epsilon\circ\varphi(a)
(\id\otimes q_E)\circ\epsilon(c)
\\&\quad=0,
\end{align*}
because
\begin{align*}
(\id\otimes q_E)\circ\bar\epsilon\circ\varphi(a)
&=(\id\otimes q_E)\circ(\varphi\otimes\id)\circ\delta(a)
\\&=(\varphi\otimes\id)\circ(\id\otimes q_E)\circ\delta(a)
\\&=(\varphi\otimes\id)(0).
\end{align*}
Thus $b\varphi(a)c\in \JJ_\epsilon(E)$.

(iii)
By \corref{J=J(E)}
it suffices to show that
$E\bigl(\varphi_*(\JJ_\delta(E))\bigr)=E\bigl(\JJ_{\epsilon}(E)\bigr)$,
and 
since $E\subset \EE(\JJ_{\epsilon}(E))$,
it further suffices to show that
$\EE(\varphi_*(\JJ_\delta(E)))\subset E$:
if $f\in \EE(\varphi_*(\JJ_\delta(E)))$, then for all $d,d'\in D$ and $a\in \JJ_\delta(E)$ we have
\begin{align*}
0
&=f\cdot(d\varphi(a)d')
\\&=df\cdot (\varphi(a))d'
\quad\text{(since $\bar\epsilon$ is trivial on $D$)}
\\&=d\varphi(f\cdot a)d'
\quad\text{(since $\varphi$ is equivariant)},
\end{align*}
and hence $f\cdot a=0$ since 
$\varphi$ is faithful and
$D$ is nondegenerate in $M(C)$.
Thus $f\in \EE(\JJ_\delta(E))=E$.
\end{proof}

\begin{lem}\label{ampliate}
Let $(A,\delta)$ be a coaction,
let $E$ be a large ideal of $B(G)$
such that $\EE(\JJ_\delta(E))=E$,
let $D$ be a $C^*$-algebra,
and
let $\id\otimes\delta$ be the tensor-product coaction on $D\otimes A$.
Then:
\begin{enumerate}
\item
The ideal $D\otimes \JJ_\delta(E)$ of $D\otimes A$ is 
small, and
is contained in $\JJ_{\id\otimes\delta}(E)$.

\item
If $D\otimes\JJ_\delta(E)=\JJ_{\id\otimes\delta}(E')$
for some large ideal $E'$, then
$D\otimes\JJ_\delta(E)=\JJ_{\id\otimes\delta}(E)$.

\item
$\JJ_{\id\otimes\delta}(E)=\ker(\id_D\otimes Q_E)$,
where $Q_E:A\to A^E$ is the quotient map,
so
$D\otimes\JJ_\delta(E)=\JJ_{\id\otimes\delta}(E)$ if and only if
the sequence
\[
\xymatrix{
0 \to
D\otimes \JJ_\delta(E) \to
D\otimes A \to
D\otimes A^E \to
0
}
\]
is exact.
\end{enumerate}
\end{lem}

\begin{proof}
For the first two parts, we verify the hypotheses of \lemref{equivariant},
including those of part (iii),
with
$(C,\epsilon)=(D\otimes A,\id\otimes\delta)$,
$\varphi=1\otimes\id_A$,
and $D$ in \lemref{equivariant} replaced by $D\otimes 1$.
The map $1\otimes \id_A:A\to M(D\otimes A)$
is $\delta-(\id\otimes\delta)$ equivariant, nondegenerate, and faithful,
$D\otimes A=\clspn\{(D\otimes 1)(1\otimes A)\}$,
$D\otimes 1$ is a nondegenerate $C^*$-subalgebra of $M(D\otimes A)$,
and $(\id\otimes\delta)(d\otimes 1)=d\otimes 1\otimes 1$ for all $d\in D$.

For (iii),
note that
\[
\JJ_{\id\otimes\delta}(E)
=\ker(\id_D\otimes\id_A\otimes q_E)\circ(\id_D\otimes\delta).
\]
Since
\[
\ker(\id_A\otimes q_E)\circ\delta=\JJ_\delta(E)=\ker Q_E,
\]
there is an injective homomorphism $\tilde\delta$ making the diagram
\[
\xymatrix@C+30pt{
D\otimes A \ar[r]^-{\id\otimes\delta} \ar[d]_{\id\otimes Q_E}
&D\otimes A\otimes C^*(G) \ar[d]^{\id\otimes\id\otimes q_E}
\\
D\otimes A^E \ar@{-->}[r]_-{\id\otimes\tilde\delta}
&D\otimes A\otimes C^*_E(G)
}
\]
commute.
Therefore
$\JJ_{\id\otimes\delta}(E)=\ker(\id_D\otimes Q_E)$.
\end{proof}

\begin{thm}\label{counterexample}
Let $G$ be nonamenable and residually finite, e.g., $\F_2$,
and consider the tensor product coaction $(C^*(G)\otimes C^*(G),\id\otimes\delta_G)$.
Then the ideal $C^*(G)\otimes\ker\lambda$
is 
small, but
is not of the form $\JJ(E)$,
and hence the associated quotient coaction is not $E$-determined,
for any large ideal $E$ of $B(G)$.
\end{thm}

\begin{proof}
By \cite[Proposition~3.7.10]{BO}, the sequence
\[
\xymatrix{
0 \to
C^*(G)\otimes \ker\lambda \to
C^*(G)\otimes C^*(G) \to
C^*(G)\otimes C^*_r(G) \to
0
}
\]
is not exact.
We have
\[
\ker\lambda=\JJ_{\delta_G}(B_r(G))\midtext{and}
\EE(\JJ_{\delta_G}(B_r(G)))=B_r(G),
\]
so the result follows from \corref{ampliate}.
\end{proof}

\begin{rems}
(1)
It follows from \cite[Lemma~1.16(a)]{fullred} that the coaction $(D\otimes_{\max} A,\id\tilde\otimes\delta)$ is maximal.
For the case $(A,\delta)=(C^*(G),\delta_G)$, Buss and Echterhoff \cite[Example~5.4]{BusEch} have shown that whenever the canonical map $D\otimes_{\max} C^*(G)\to D\otimes C^*(G)$ is not faithful then the coaction $(D\otimes C^*(G),\id\otimes\delta_G)$ is not $E$-determined from its maximalization for any large ideal $E$ of $B(G)$.

(2)
\thmref{counterexample} shows that the map $\JJ$ from large ideals of $B(G)$ to small 
ideals of $A$ is not surjective in general.
It is easy to see that $\JJ$ is also generally not injective, either. For the most extreme source of examples of this, let $\delta$ be a coaction that is both maximal and normal, and let $G$ be nonamenable. Then $\{0\}$ is the only small ideal of $A$, but there can be many large ideals of $B(G)$ --- indeed, it follows from a result of \cite{okayasu} that $B(\F_n)$ has a continuum of such ideals whenever $n\ge 2$ --- see the discussion preceding \propref{ok} below for further discussion of this.

(3)
Similarly to \corref{ampliate}, if $(B,\alpha)$ is an action then the ideal
\[
(i_G)_*(\pann E)=\clspn\{(B\rtimes_\alpha G)i_G(\pann E)(B\rtimes_\alpha G)\}
\]
of $B\rtimes_\alpha G$ is 
small,
is contained in 
$\JJ_{\what\alpha}(E)$,
and 
is of the form $\JJ_{\what\alpha}(E')$ for some coaction ideal $E'$ if and only if
it in fact equals $\JJ_{\what\alpha}(E)$.
Since we have no application of this result in mind, we omit the proof --- it follows from \propref{equivariant} similarly to \corref{ampliate}.
This result is not quite a generalization of \corref{ampliate}, because $B\rtimes_\iota G\cong B\otimes_{\max} C^*(G)$, not $B\otimes C^*(G)$ (where $\iota$ denotes the trivial action).
\end{rems}

\section{$E$-determined twice}\label{twice}

Suppose that $(A,\delta)$ is a slice proper maximal coaction for which every 
small ideal is of the form $\JJ(E)$ for some large ideal $E$ of $B(G)$.
Let $J_1\subset J_2$ be two 
small ideals of $A$, so that by assumption we have
$J_i=\JJ_\delta(E_i)$ for some $E_1,E_2$.
By our general theory, we can assume without loss of generality that
\[
E_i=\EE(J_i):=\{f\in B(G):(x\cdot f\cdot y)\cdot J_i=\{0\} \text{ for all }x,y\in G\}.
\]
Then $E_1\supset E_2$, and there exist:
\begin{enumerate}
\item coactions $\delta_i$ of $G$ on the quotients $A_i=A/J_i$,
\item $\delta-\delta_i$ equivariant surjections $Q_i:A\to A_i$, and
\item a $\delta_1-\delta_2$ equivariant surjection $Q_{12}$ making the diagram
\[
\xymatrix{
A \ar[r]^-{Q_1} \ar[dr]_{Q_2}
&A_1 \ar[d]^{Q_{12}}
\\
&A_2
}
\]
commute.
\end{enumerate}

\begin{q}\label{12}
With the above notation, 
is the coaction $(A_2,\delta_2)$ $E$-determined from $(A_1,\delta_1)$ for some large ideal $E$ of $B(G)$,
equivalently,
is the ideal $\ker Q_{12}$ of $A_1$ of the form $\JJ_{\delta_1}(E)$ for some $E$?
\end{q}

It 
seems difficult to answer \qref{12} ---
if we think the answer is yes, then we should presumably find an appropriate $E$. What could it be?
Certainly it could not be $E_1$, because this has nothing to do with $E_2$.
On the other hand,
in general it is not $E_2$, either, as
we will show 
in \propref{pq}.

\begin{notn}
In the following \lemref{E} and \corref{equal E}, we denote the weak* closed span of a subset $S\subset B(G)$ by $[S]$.
\end{notn}

\begin{lem}\label{E}
With the above notation, for any large ideal $E$ of $B(G)$, we have:
\begin{align}
\JJ_{\delta_1}(E)&=Q_1(\JJ_\delta([E_1E]))\label{first}\\
\ker Q_{12}&=Q_1(\JJ_\delta(E_2))\subset \JJ_{\delta_1}(E_2).\label{second}
\end{align}
\end{lem}

\begin{proof}
For \eqref{first},
since $Q_1$ is a surjective linear map, it suffices to observe that for $a\in A$, we have
\begin{align*}
&Q_1(a)\in \JJ_{\delta_1}(E)
\\&\quad\Leftrightarrow 0=E\cdot Q_1(a)=Q_1(E\cdot a)\quad\text{(by equivariance)}
\\&\quad\Leftrightarrow E\cdot a\subset \ker Q_1=\JJ_\delta(E_1)
\\&\quad\Leftrightarrow 0=E_1\cdot E\cdot a=[E_1E]\cdot a
\\&\quad\Leftrightarrow a\in \JJ_\delta([E_1E]).
\end{align*}

For \eqref{second}, we first consider the equality:
since $Q_1$ is surjective and $Q_2=Q_{12}\circ Q_1$,
\[
\ker Q_{12}=Q_1(\ker Q_2)=Q_1(\JJ_\delta(E_2)).
\]
For the other part, as $[E_1E_2]\subset E_2$, we have $\JJ_\delta(E_2)\subset \JJ_\delta([E_1E_2])$,
and so the inclusion $Q_1(\JJ_\delta(E_2))\subset \JJ_{\delta_1}(E_2)$ now follows from \eqref{first} with $E=E_2$.
\end{proof}

\begin{cor}\label{equal E}
For a large ideal $E$ of $B(G)$, 
if 
$E_1E$ has weak* dense span in $E_2$,
then
$\ker Q_{12}=\JJ_{\delta_1}(E)$,
and hence the quotient $(A_2,\delta_2)$ of $(A_1,\delta_1)$ is $E$-determined from $(A_1,\delta_1)$.
\end{cor}

\begin{proof}
By \lemref{E}, we have $\ker Q_{12}=\JJ_{\delta_1}(E)$ if and only if $Q_1(\JJ_\delta(E_2))=Q_1(\JJ_\delta([E_1E]))$.
Since $E_1$ contains both $E_2$ and $[E_1E]$, and since the map $\JJ$ is inclusion-reversing, we see that $\ker Q_1=\JJ_\delta(E_1)$ is contained in both $\JJ_\delta(E_2)$ and $\JJ_\delta([E_1E])$, and hence $Q_1(\JJ_\delta(E_2))=Q_1(\JJ_\delta([E_1E]))$ if and only if $\JJ_\delta(E_2)=\JJ_\delta([E_1E])$.
\end{proof}

The above lemma leads us to another question:

\begin{q}\label{E_1E}
For large ideals $E_1\supset E_2$ of $B(G)$,
does there exist a large ideal $E$ of $B(G)$ such that
$E_1E$ has weak* dense span in $E_2$?
\end{q}

By \corref{equal E}, an affirmative answer to \qref{E_1E} would imply one for \qref{12}.

Note that, even with all our restrictions on the ideals $E$, the map $\JJ$ from the large ideals of $B(G)$ to the small 
ideals of $A$ is not injective,
and so
we are lead to suspect that the converse of \corref{equal E} does not hold.
That being said, let us consider the special case $(A,\delta)=(C^*(G),\delta_G)$.
Since for this maximal coaction the map $\JJ$ from large ideals of $B(G)$ to small 
ideals is injective (in fact, is bijective),
we can conclude:

\begin{cor}\label{discrete}
With the above notation, 
the quotient $(C^*_{E_2}(G),\delta_G^{E_2})$ of $(C^*_{E_1}(G),\delta_G^{E_1})$ is $E$-determined from $(C^*_{E_1}(G),\delta_G^{E_1})$ if and only if $E_1E$ has weak* dense span in $E_2$.
\end{cor}

It is interesting to consider the 
special case $E=E_1=E_2$, since it makes a connection with the $C^*$-bialgebra structure.

But first, another definition:

\begin{defn}\label{E-normal}
A coaction $(A,\delta)$ is \emph{$E$-normal} if $(\id\otimes q_E)\circ\delta$ is faithful.
\end{defn}

\begin{ex}
A coaction is normal in the usual sense if and only if it is $B_r(G)$-normal in the above sense.
At the other extreme, every coaction is $B(G)$-normal, because $q_{B(G)}$ is the identity map.
Note that every normal coaction is $B_r(G)$-determined from its maximalization,
and every maximal coaction is $B(G)$-determined from itself.
However, we will show in \propref{ok} that in general a coaction that is $E$-determined from its maximalization need not be $E$-normal.
\end{ex}

Recall that the ``canonical'' comultiplication $\Delta_G^E$ on $C^*_E(G)$ is 
defined as the unique homomorphism making the diagram
\[
\xymatrix{
C^*_E(G) \ar[r]^-{\delta_G^E} \ar[dr]_(.4){\Delta_G^E}
&M(C^*_E(G)\otimes C^*(G)) \ar[d]^{\id\otimes q_E}
\\
&M(C^*_E(G)\otimes C^*_E(G))
}
\]
commute.

\begin{cor}\label{bialg}
If $E$ is a large ideal of $B(G)$, then the following are equivalent:
\begin{enumerate}
\item $(C^*_E(G),\delta_G^E)$ is $E$-determined from $(C^*_E(G),\delta_G^E)$;
\item $E^2$ has weak* dense span in $E$;
\item the canonical comultiplication $\Delta_G^E$ on $C^*_E(G)$ is faithful;
\item $(C^*_E(G),\delta_G^E)$ is $E$-normal in the sense of \defnref{E-normal}.
\end{enumerate}
\end{cor}

\begin{proof}
(i) $\iff$ (ii) by \corref{discrete} with $E=E_1=E_2$.
Since $E$ is the dual of $C^*_E(G)$,
$E^2$ has weak* dense span in $E$ if and only if the preannihilator $\pann(E^2)$ in $C^*_E(G)$ is $\{0\}$,
equivalently there is no nonzero $c\in C^*_E(G)$ with $(fg)(c)=0$ for all $f,g\in E$.
Since
\[
(fg)(c)=(f\otimes g)\circ\Delta_G^E(c)\quad\text{for $f,g\in E$},
\]
and the elementary tensors $f\otimes g$ separate points in $C^*_E(G)\otimes C^*_E(G)$,
we conclude (ii) $\iff$ (iii).
Finally, (iii) $\iff$ (iv) follows immediately from the definition of $E$-normality.
\end{proof}

\section{$L^p$}\label{lp}

In this section we illustrate the preceding discussion in the case of ideals of $B(G)$ determined by the $L^p$ spaces.
Note that for $1\le p<\infty$ the intersection
$L^p(G)\cap B(G)$ is a $G$-invariant ideal of $B(G)$.

\begin{defn}\label{Ep ideal}
For $1\le p<\infty$ we let $E_p=E_p(G)$ denote the weak* closure of $L^p(G)\cap B(G)$.
\end{defn}

Since $L^p(G)\cap B(G)$ is a $G$-invariant ideal of $B(G)$, so is $E_p$.
Since $C_c(G)\subset L^p(G)$, $E_p$ contains 
$C_c(G)\cap B(G)$, so $E_p$ is a large ideal of $B(G)$, i.e., contains $B_r(G)$.

Consider the maximal coaction $(C^*(G),\delta_G)$.
Our general theory shows that every large quotient coaction of $(C^*(G),\delta_G)$ is $E$-determined for some large ideal $E$ of $B(G)$.
We do not know the answer to \qref{12} even in this setting, 
but we can at least 
give some information
when we restrict the ideals of $B(G)$ to be of the form $E_p$:

\begin{prop}\label{pq}
Let $\infty>p>q\ge 1$,
so that $E_p\supset E_q$,
where the ideals $E_p$ are defined in \defnref{Ep ideal}.
Then the weak* closed span of $E_pE_r$ is contained in $E_q$, where
\[
\frac 1r+\frac 1p=\frac 1q.
\]
\end{prop}

\begin{proof}
Since multiplication in $B(G)$ is separately weak* continuous,
it suffices to observe that, by a routine application of H\"older's inequality,
\[
L^p(G)L^r(G)\subset L^q(G).
\qedhere
\]
\end{proof}

\begin{rem}
We cannot conclude from the above proof that $E_pE_r$ has weak* dense span in $E_q$, because we can't take roots in $B(G)$; more precisely, we do not see how to prove that the span of $E_pE_r$ in $B(G)$ is weak* dense in $E_q$.
\end{rem}

It has been attributed to independent work of Higson, Ozawa, and Okayasu (see \cite[Remark~4.5]{BO}, \cite[Corollary~3.7]{okayasu}) that (in our notation) for $2\le d<\infty$ and $\infty>p>q\ge 2$, the canonical quotient map of $C^*_{E_p}(\F_d)$ onto $C^*_{E_q}(\F_d)$ is not faithful, equivalently $E_p\ne E_q$.
On the other hand, \cite[Proposition~2.11]{BO} (see also \cite[Proposition~4.2]{graded}) implies that $E_p=E_q$ for all $2\ge p>q\ge 1$.
This leads to:

\begin{prop}\label{ok}
For $2\le d<\infty$ and $\infty>p>2$, the canonical comultiplication $\Delta_{\F_d}^{E_p}$ on $C^*_{E_p}(\F_d)$ is not faithful, and the coaction $(C^*_{E_p}(\F_d),\delta_{\F_d}^{E_p})$, although it is $E_p$-determined from its maximalization, is not $E_p$-normal in the sense of \defnref{E-normal}.
\end{prop}

\begin{proof}
It follows from 
\propref{pq} that the weak* closed span of $E_p^2$ is 
contained in $E_{p/2}$, and hence is different from $E_p$
by the discussion preceding \corref{ok}. Thus the result follows from \corref{bialg}.
\end{proof}

\begin{q}
The above discussion of the conditions listed in \corref{bialg} as they relate to the ideals $E_p(G)$ should be carried out for some other well-known large ideals of $B(G)$, namely the weak* closure $E_0(G)$ of $C_0(G)\cap B(G)$ and the ideal $E$ orthogonal to the almost periodic functions $AP(G)$ (see \cite[Remark~4.3 (3)]{graded}).
For example, it would be interesting to know whether, in each of these cases, the square of the ideal is weak* dense in the ideal itself.
\end{q}


\providecommand{\bysame}{\leavevmode\hbox to3em{\hrulefill}\thinspace}
\providecommand{\MR}{\relax\ifhmode\unskip\space\fi MR }
\providecommand{\MRhref}[2]{%
  \href{http://www.ams.org/mathscinet-getitem?mr=#1}{#2}
}
\providecommand{\href}[2]{#2}

\end{document}